\documentclass[11pt, lettersize]{article}
\usepackage{amsthm, amssymb, amsmath}
\usepackage{algorithm}
\usepackage{verbatim, float}
\usepackage{mathrsfs}
\usepackage{graphics}
\usepackage[usenames,dvipsnames]{pstricks}
\usepackage{epsfig}
\usepackage{cases}
\usepackage{algorithmic}
\usepackage[figuresright]{rotating}
\usepackage[top=2cm, bottom=2.5cm, left=2cm, right=2cm]{geometry}

\theoremstyle{plain}
\newtheorem{theorem}{Theorem}
\newtheorem{lemma}{Lemma}
\newtheorem{conjecture}{Conjecture}

\newtheorem{proposition}{Proposition}
\newtheorem{observation}{Observation}
\theoremstyle{definition}
\newtheorem{definition}{Definition}

\newtheorem{remark}{Remark}

\newcommand{\A}{{\mathcal A}}
\newcommand{\B}{{\mathcal B}}

\newcommand{\D}{{\mathcal D}}

\newcommand{\tvG}{{\theta_{\text{v}}(G)}}
\newcommand{\teG}{{\theta_{\text{e}}(G)}}
\newcommand{\teH}{{\theta_{\text{e}}(H)}}
\newcommand{\ttG}{{\theta_{\triangle}(G)}}
\newcommand{\ttee}{\theta_{\triangle}}
\newcommand{\tKtG}{{\theta_{K_t}(G)}}

\newcommand{\tv}{{\theta_{\text{v}}}}
\newcommand{\te}{{\theta_{\text{e}}}}
\newcommand{\ttr}{{\theta_{\triangle}}}
\newcommand{\tKt}{{\theta_{K_t}}}

\newcommand{\caze}[2]{\textbf{Case {#1}:} \textit{#2}}
\newcommand{\sizeof}[1]{\left\lvert{#1}\right\rvert}
\newcommand{\st}{\colon\,}

\newcommand{\cee}{\mathcal{C}}
\newcommand{\fee}{\mathcal{F}}

\newcommand{\iw}{i_{w,t}}
\newcommand{\pw}{p_{w,t}}
\newcommand{\ints}{\mathbb{Z}}

\newcommand{\kay}{\mathcal{K}}
\newcommand{\lay}{\mathcal{L}}

\newcommand{\optpair}{\textsf{optpair}}
\newcommand{\floor}[1]{\left\lfloor #1 \right\rfloor}
\newcommand{\ceil}[1]{\lceil #1 \rceil}
\newcommand{\iso}{\cong}
\newcommand{\nonneg}{\ints_{\geq 0}}
\DeclareMathOperator{\KCC}{KCC}
\DeclareMathOperator{\cost}{cost}
\DeclareMathOperator{\valu}{val}

\newcommand{\erdos}{{Erd\"{o}s}}
\newcommand{\posa}{{P\'{o}sa}}
\newcommand{\turan}{{Tur\'{a}n}}
\newcommand{\bollobas}{{Bollob\'as}}

\begin{document}

\title{On the Triangle Clique Cover and $K_t$ Clique Cover Problems}

\author{Hoang Dau \and Olgica Milenkovic \and Gregory J.~Puleo}

\date{}

\maketitle 

\begin{abstract}
  An edge clique cover of a graph is a set of cliques that covers all
  edges of the graph.  We generalize this concept to \emph{$K_t$
    clique cover}, i.e. a set of cliques that covers all complete
  subgraphs on $t$ vertices of the graph, for every $t \geq 1$.  In
  particular, we extend a classical result of \erdos, Goodman, and
  P\'{o}sa~(1966) on the edge clique cover number
  ($t = 2$), also known as the intersection number, to the case $t =
  3$. The upper bound is tight, with equality holding only for the
  Tur\'an graph $T(n,3)$. As part of the proof, we obtain
    new upper bounds on the classical intersection number,
    which may be of independent interest. We also extend an algorithm of Scheinerman
  and Trenk~(1999) to solve a weighted version
  of the $K_t$ clique cover problem on a superclass of chordal graphs.
  We also prove that the $K_t$ clique cover problem is NP-hard.
\end{abstract}




\section{Introduction}
\label{sec:intro}

A \emph{clique} in a graph $G$ is a set of vertices that induces a
complete subgraph; all graphs considered in this paper are simple and
undirected. A \emph{vertex clique cover} of a graph $G$ is a set of
cliques in $G$ that collectively cover all of its vertices. The
\emph{vertex clique cover number} of $G$, denoted $\tv(G)$, is the
minimum number of cliques in a vertex clique cover of $G$. An
\emph{edge clique cover} of a graph $G$ is a set of cliques of $G$
that collectively cover all of its edges. The \emph{edge clique cover
  number} of $G$, denoted $\te(G)$, is the minimum number of cliques
in an edge clique cover. The vertex clique cover number, which is the
same as the chromatic number of the complement graph, and the
edge clique cover number, also referred to as the \emph{intersection
  number} of a graph, have been extensively studied in the literature
(see, for instance~\cite{ErdosGoodmanPosa1966, Karp1972,
  Roberts1985}).

\begin{figure}[htb]
\centering
\includegraphics[scale=1]{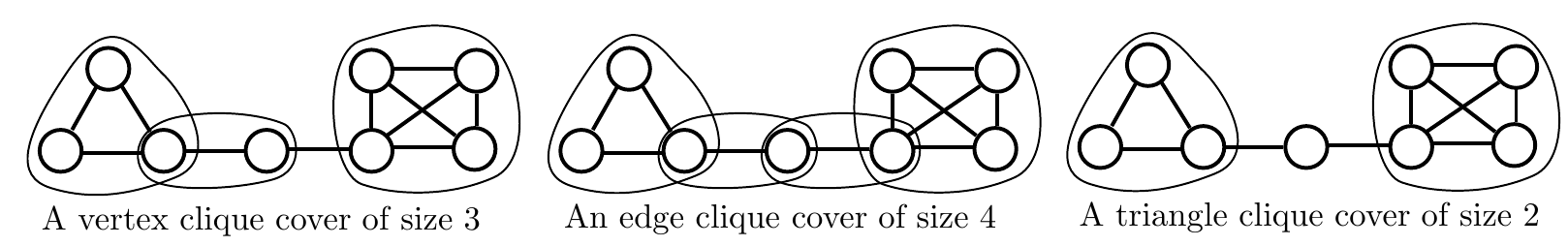}
\caption{An illustration of a minimum \emph{vertex} clique cover (left-most), a minimum \emph{edge} clique cover (middle), and a minimum \emph{triangle} clique cover (right-most) of the same graph on eight vertices.}
\label{fig:toy-example}
\end{figure}

We generalize the notions of vertex and edge clique covers by the following
definition. Note that we use the word \emph{clique} to refer to any
vertex set inducing a complete subgraph. 

\begin{definition}
  Let $t$ be a nonnegative integer. A \emph{$t$-clique} of $G$ is a clique containing exactly $t$ vertices. 
A set $\cee{}$ of cliques is a \emph{$K_t$ clique cover of $G$ if for every $t$-clique $S \subset V(G)$ there is a clique $Q \in \cee$ that covers $S$ (i.e., $S$ is a subgraph of $Q$).}
\end{definition}

A $K_1$ clique cover is simply a vertex
clique cover, while a $K_2$ clique cover is an edge clique cover. We
refer to a $K_3$ clique cover as a \emph{triangle clique cover}.  The
\emph{$K_t$ clique cover number}, denoted $\tKt$, and the
\emph{triangle clique cover number}, denoted $\ttr$, are also defined
accordingly.  We illustrate these concepts in
Fig.~\ref{fig:toy-example}.

For positive integers $n$ and $k$, the \emph{Tur\'{a}n graph} $T(n,k)$
is defined to be the complete $k$-partite graph on $n$ vertices whose
part sizes differ by at most $1$. Tur\'{a}n graphs frequently arise as
extremal graphs for various graph parameters. We pose the following conjecture, which
states that Tur\'{a}n graphs are the unique extremal graphs for the
$K_t$ clique cover problem.
\begin{conjecture} 
  \label{conjecture}
  If $n$ and $t$ are positive integers, then for every $n$-vertex graph $G$,
  \begin{equation} \tKtG \leq \tKt(T(n,t)). \label{eq:conjecture} \end{equation}
  Equality holds if and only if $G \iso T(n,t)$.
\end{conjecture}
As motivation for Conjecture~\ref{conjecture}, we now briefly consider
the cases $t=1$ and $t=2$.

When $t = 1$, it is obvious that $\theta_{K_1}(G) = \tvG \leq n$ for
every graph $G$ on $n$ vertices.  Equality clearly holds if and only
if $G$ has no edges; a graph with no edges is the trivial Tur\'{a}n
graph $T(n,1)$. \erdos, Goodman, and
P\'{o}sa~\cite{ErdosGoodmanPosa1966} proved that
$\theta_{K_2}(G) = \teG \leq \lfloor \frac{n^2}{4} \rfloor$ for every
graph $G$, with equality holding if and only if $G$ is the Tur\'an
graph $T(n,2) = K_{\floor{n/2}, \ceil{n/2}}$.

\begin{table}
  \centering
  \begin{tabular}{|l|l|l|l|l|}
    \hline
    & $t = 1$ & $t = 2$ & $t = 3$ & $t > 3$\\
    \hline
    \eqref{eq:conjecture} holds & YES & YES & YES & Open\\
    \hline
    Equality & $G=T(n,1)$ & $G=T(n,2)$ & $G=T(n,3)$ & Open\\
    \hline
    Proof & Trivial & Erd\"{o}s, Goodman, and Posa~\cite{ErdosGoodmanPosa1966} & 
                                                                                 This paper & Open\\
    \hline
  \end{tabular}
  \caption{Status of Conjecture~\ref{conjecture}. }
  \label{table:conjstatus}
\end{table}

In Sections~\ref{sec:mindeg-lovasz} and~\ref{sec:upper}, we prove the
$t=3$ case of Conjecture~\ref{conjecture}. At the end of
Section~\ref{sec:upper}, we also discuss the connections between
Conjecture~\ref{conjecture} and a theorem of Lehel~\cite{Lehel1982}
about covering edges in hypergraphs. The status of
Conjecture~\ref{conjecture} is summarized in
Table~\ref{table:conjstatus}.

We also consider a natural weighted version of the $K_t$ clique cover problem in Section~\ref{sec:weightkt}, in which each $t$-clique $S$ is assigned a nonnegative integer weight $w_S$; we seek a smallest multiset of cliques such that each $t$-clique $S$ is covered at least $w_S$ times. We give a
polynomial-time algorithm to solve this problem on a superclass of chordal graphs, extending a result of Scheinerman and
Trenk~\cite{scheinerman-trenk}. The specific class of graphs
for which our algorithm works, which we call \emph{semichordal graphs},
is defined and discussed in Section~\ref{sec:superclass}.


It was shown by Orlin~\cite{Orlin1977} and by Kou, Stockmeyer, and
Wong~\cite{KouStockmeyerWong78} that determining $\teG$ is an
NP-complete problem.  Their idea is to reduce the problem of
determining $\tvG$, which was known to be NP-complete, to the problem
of determining $\teG$. In Section~\ref{sec:complexity}, we generalize
the reduction used in~\cite{KouStockmeyerWong78} to show, by
induction, that determining $\tKtG$ is an NP-complete problem, for any
constant $t \geq 2$, by reducing the problem of determining
$\theta_{K_{t-1}}(G)$ to the problem of determining $\theta_{K_t}(G)$.

Our study of triangle clique covers and $K_t$ clique covers is motivated by the recent developments in the literature of \emph{community detection} (a.k.a network clustering). Complex networks such as social networks and biological networks often exhibit strong community structures in which nodes are clustered into different communities based on dense intra-community connections and sparse inter-community connections. Cliques or clique-like substructures in the network are often used to model such communities which usually represent groups of people with common affiliations or interests in social networks, or disciplines in the citation networks, or functional modules in the protein-protein interaction networks \cite{Palla-etal-2005, Bansal-etal-2004, Leskovec-etal-2010}. Most community detection algorithms are based on pairwise connections among nodes in the network. Nevertheless, recent work~\cite{Agarwal-etal-2005, Benson-etal-2016, Li-etal-2017, Li-etal-NIPS-2017} has revealed the importance of \emph{motifs}, that is, subgraphs of the graph that appear with a frequency exceeding the one predicted through certain random models. Examples of motifs include triangles and small cliques and near-cliques, the former arising due to the principle of \emph{triadic closure}, which assumes that two people having a common friend will be more likely to be connected~\cite{granovetter1973, Bianconi-etal-2014}. The problem of covering triangles with cliques corresponds to the problem of network clustering with triangle motifs in the ideal setting where each community is a clique.  

 
\section{A Minimum-Degree Version of a Theorem of Lov\'asz}\label{sec:mindeg-lovasz}
In this section, we obtain results similar to the following theorem of
Lov\'asz~\cite{lovasz-cover}. In the next section, we will apply these
results to obtain an upper bound on $\ttG$.
\begin{theorem}[Lov\'asz~\cite{lovasz-cover}]\label{thm:lovasz}
  Let $G$ be an $n$-vertex graph, and let
  $k = {n \choose 2} - \sizeof{E(G)}$.  If $t$ is the greatest integer
  such that $t^2-t \leq k$, then $\te(G) \leq k+t$.
\end{theorem}
In this section, we obtain a variant of Theorem~\ref{thm:lovasz} by
strengthening the hypothesis to include a lower bound on $\delta(G)$,
the minimum degree of $G$, rather than just a lower bound on
$\sizeof{E(G)}$. We start with a basic result and then strengthen its
weakest nontrivial case.
\begin{lemma}\label{lem:basicdelta}
  For any graph $G$, $\te(G) \leq (n - \delta(G)) + \frac{n(n-\delta(G)-1)}{2}$.
\end{lemma}
\begin{proof}
  We mimic the proof of Theorem~\ref{thm:lovasz}.  Let $A_1$ be a
  maximum clique in $G$, and for $i > 1$, let $A_i$ be a maximum
  clique in $G-(A_1 \cup \cdots \cup A_{i-1})$.  Set
  $a_i = \sizeof{A_i}$. Let $p$ be the largest index for which $a_i > 0$.  As each vertex of $A_p$ has at least one non-neighbor
  in $A_j$ for $j < p$, we have $p \leq n - \delta(G)$.  For each
  $v \in A_i$ and each $j < i$, let
  $S_{v,j} = \{v\} \cup (N(v) \cap A_j)$.  Each set $S_{v,j}$ is
  clearly a clique.  Let $\fee$ be the set consisting of all cliques
  $A_i$ together with all the cliques $S_{v,j}$ where $v \in A_i$ and
  $j < i$. The set $\fee$ covers all edges of $G$, and we have
  \[ \te(G) \leq \sizeof{\fee} \leq p + \sum_{i=1}^{p}(i-1)a_i. \]
  Subject to the constraints $a_1 + \ldots + a_p = n$ and
  $a_1 \geq \cdots \geq a_p$, and allowing $a_i$ to take fractional
  values, the sum $\sum_{i=1}^p (i-1)a_i$ is clearly maximized when
  $a_1 = \cdots = a_p = n/p$. Hence,
  \[ \te(G) \leq p + \frac{n}{p}\sum_{i=1}^p(i-1) = p + \frac{n(p-1)}{2}. \]
  As $p \leq n-\delta(G)$, the conclusion follows.
\end{proof}
Note that a lower bound on $\delta(G)$ indirectly gives an upper bound on $\te(G)$ via Theorem~\ref{thm:lovasz}, since we have $k \geq \frac{n(n-\delta(G)-1)}{2}$. However, the upper bound in Lemma~\ref{lem:basicdelta} is often sharper than the bound guaranteed this way in Theorem~\ref{thm:lovasz}; for example, when $n=12$ and $\delta(G) = 9$, Theorem~\ref{thm:lovasz} gives an upper bound of $\te(G) \leq 16$, while Lemma~\ref{lem:basicdelta} gives an upper bound of $15$.

When $\delta(G) = n/2$, Lemma~\ref{lem:basicdelta} gives $\te(G) \leq n^2/4$, which
is sharp when $G = K_{n/2, n/2}$. We wish to obtain a sharper bound when $\delta(G)$
is slightly larger than $n/2$.

\begin{lemma}\label{lem:plushalf_plusone}
The following upper bounds hold for the edge clique cover number of a graph $G$.
\[
\te(G) \leq
\begin{cases}
\frac{n^2}{4} - \frac{n}{2} + \frac{1}{4},&\text{ if } \delta(G) = (n+1)/2,\\
\frac{n^2}{4} - n + 2, &\text{ if } \delta(G) = n/2+1.
\end{cases}
\]
\end{lemma}
\begin{proof}
  Define $A_1, \ldots, A_p$ and $\fee$ as in the proof of Lemma~\ref{lem:basicdelta}.
  With $a_i = \sizeof{A_i}$, we have the bound
  \[ \te(G) \leq \sizeof{\fee} \leq p + \sum_{i=1}^{p}(i-1)a_i. \]
  
\textbf{Case 1:} $p < n-\delta(G)$. If $\delta(G) = (n+1)/2$, then repeating the argument in Lemma~\ref{lem:basicdelta} yields
  \[ \te(G) \leq p + \frac{n(p-1)}{2} \leq \frac{n^2}{4} - \frac{3n}{4} - \frac{3}{2} < \frac{n^2}{4} - \frac{n}{2} + \frac{1}{4}. \]
If $\delta(G) = n/2 + 1$, then similarly,
  \[ 
	\te(G) \leq p + \frac{n(p-1)}{2} \leq \frac{n^2}{4} - n - 2
	< \frac{n^2}{4} - n + 2. \]

\textbf{Case 2:} $p = n-\delta(G)$. We now exploit the integrality of $a_i$, which we ignored in the proof of Lemma~\ref{lem:basicdelta}. If $\delta(G) = (n+1)/2$ then $p = (n-1)/2$ and the sum $\sum_{i=1}^p (i-1)a_i$ is maximized, subject to $a_1 + \ldots + a_p = n$  and $a_1 \geq \cdots \geq a_p$, by the sequence with $a_1 = 3$ and $a_i = 2$ for $i \geq 2$.
  Hence, we obtain the upper bound
\[ 
\te(G) \leq p + 2\sum_{i=2}^p (i-1) = p + p(p-1) = \frac{n^2}{4} - \frac{n}{2} + \frac{1}{4}. 
 \]
Similarly, when $\delta(G) = n/2 + 1$ and $p = n/2-1$,
the sum $\sum_{i=1}^p (i-1)a_i$ is maximized by the sequence with $a_1 = a_2 = 3$ and $a_i = 2$ for $i \geq 3$. We obtain the upper bound
\[ 
\te(G) \leq p + 3 + 2\sum_{i=3}^p(i-1) = p+3 + (p(p-1) - 2) = p+1+p(p-1) = \frac{n^2}{4} - n + 2. 
\]
Thus, the claimed upper bounds on $\teG$ hold in both cases.
\end{proof}

\section{An Upper Bound on $\ttG$}\label{sec:upper}
In this section, our goal is to generalize the following result of \erdos, Goodman, and P\'osa.
\begin{theorem}[\erdos--Goodman--Posa~\cite{ErdosGoodmanPosa1966}]\label{thm:EGP}
  If $G$ is an $n$-vertex graph, then $\te(G) \leq \floor{n^2/4}$.
  Equality holds if and only if $G \iso T(n,2)$.
\end{theorem}
\begin{definition}
  When $G$ is a graph and $r$ is a nonnegative integer, $k_t(G)$
  is the number of copies of $K_t$ in $G$.
\end{definition}
\begin{observation}
  For any nonnegative integer $n$,
  \[ k_3(T(n,3)) = \begin{cases}
\dfrac{n^3}{27}, &\text{ if } n \equiv 0 \pmod 3,\vspace{.5em}\\
\dfrac{(n-1)^3}{27} + \dfrac{(n-1)^2}{9}, &\text{ if } n \equiv 1 \pmod 3,\vspace{.5em}\\
\dfrac{(n+1)^3}{27} - \dfrac{(n+1)^2}{9}, &\text{ if } n \equiv 2 \pmod 3. 
\end{cases} \]
\end{observation}
\begin{observation}\label{obs:kdiff}
  For all $n \geq 3$,
  \begin{align*}
	k_3(T(n,3)) - k_3(T(n-1,3)) &= \floor{\frac{\floor{2n/3}^2}{4}}\\
	&=\begin{cases}
          \dfrac{n^2}{9}, &\text{ if $n \equiv 0 \pmod 3$}\\
          \dfrac{(n-1)^2}{9}, &\text{ if $n \equiv 1 \pmod 3$}\\
          \dfrac{n^2-n-2}{9}, &\text{ if $n \equiv 2 \pmod 3$}
	\end{cases}\\
	&\geq \frac{(n-1)^2}{9}.   
  \end{align*}
\end{observation}

\begin{theorem}
\label{thm:upper_bound}
  For any graph $G$, $\ttG \leq k_3(T(n,3))$. If equality holds,
  then $G \iso T(n,3)$.
\end{theorem}
\begin{proof}
  We use induction on $n$, with trivial base case when $n \leq 3$.
  Assume that $n > 3$ and the claim holds for smaller graphs. Let $v$
  be a vertex of minimum degree in $G$, let $G' = G-v$, and let $\cee'$
  be a smallest $K_3$ clique cover of $G'$. By the induction hypothesis,
  $\sizeof{\cee'} \leq k_3(T(n-1,3))$. We will extend $\cee'$ to a $K_3$
  clique cover $\cee$ of $G$ such that $\sizeof{\cee} \leq k_3(T(n,3))$.
  The only triangles of $G$ not yet covered by $\cee'$ are the triangles
  that contain $v$.

  Let $H$ be the subgraph of $G$ induced by $N(v)$, and
  let $\fee$ be a smallest \emph{edge} clique cover of $H$. By adding
  $v$ to each clique in $\fee$, we obtain a set of cliques $\fee_1$
  covering every triangle that contains $v$.  Thus,
  $\fee_1 \cup \cee'$ is a triangle edge cover in $G$. It therefore
  suffices to show that $\teH \leq k_3(T(n,3)) - k_3(T(n-1,3))$, and
  this is what we show next. We split the proof into cases according to $d(v)$.

  \caze{1}{$d(v) \leq \floor{2n/3}$.} In this case,
  $\sizeof{V(H)} \leq 2n/3$, so by Theorem~\ref{thm:EGP} and
  Observation~\ref{obs:kdiff}, we have
  \[ \te(H) \leq \floor{\sizeof{V(H)}^2/4} \leq
    \floor{\floor{2n/3}^2/4} = k_3(T(n,3)) - k_3(T(n-1,3)), \] as
  desired. If $\ttG = k_3(T(n,3))$, then equality must hold throughout
  the above inequality, and in particular we must have
  $\te(H) = \sizeof{V(H)}^2/4 = \floor{\floor{2n/3}^2/4}$. By
  Theorem~\ref{thm:EGP}, this implies that
  $H \iso T(\floor{2n/3}, 2)$. Furthermore, $\ttG = k_3(T(n,3))$
  requires that $\sizeof{\cee'} = \ttee(G-v) = k_3(T(n-1,3))$, so by the
  induction hypothesis, we have $G' \iso T(n-1,3)$. Thus, $G$ is
  obtained from $T(n-1,3)$ by adding a new vertex adjacent to
  $\floor{2n/3}$ vertices inducing a complete bipartite graph. This
  implies that $G \iso T(n,3)$.

Since $v$ was a vertex of minimum degree, every $w \in N(v)$ satisfies $d(w) \geq d(v)$. At most $n-d(v)$ of those neighbors lie outside $N(v)$, so for all $w \in V(H)$, we have
  \[ d_H(w) \geq d(v) - (n-d(v)) = 2d(v)-n. \]
  Thus, $\delta(H) \geq 2d(v)-n$. This inequality will be used in the subsequent cases.
   
  \caze{2}{$d(v) \geq 2n/3 + 1$.} 
	Lemma~\ref{lem:basicdelta} yields
  \[ 
	\te(H) \leq 
	(d(v) - \delta(H)) + \frac{d(v)(d(v) -
    \delta(H)-1)}{2} \leq \frac{n(d(v)+2) - d(v)(d(v)+3)}{2}. 
	\]
  If $d(v) \geq (2n+4)/3$, then this implies that 
  \[
	\te(H) \leq \frac{n\Big(\frac{2n+4}{3}+2\Big) - \frac{2n+4}{3}\Big(\frac{2n+4}{3}+3\Big)}{2} = \frac{(n-1)^2-27}{9}
	< \frac{(n-1)^2}{9} \leq k_3(T(n,3)) - k_3(T(n-1,3)),
	\]
	and we are done. Similarly, if $d(v) = (2n+3)/3$, then $n \equiv 0 \pmod{3}$, so we again
  have 
\[
\te(H) \leq \frac{n\Big(\frac{2n+3}{3}+2\Big) - \frac{2n+3}{3}\Big(\frac{2n+3}{3}+3\Big)}{2} = \frac{2n^2-3n-36}{18}
< \frac{n^2}{9} = k_3(T(n,3)) - k_3(T(n-1,3)).
\]
  \caze{3}{$d(v) = (2n+1)/3$ or $d(v) = (2n+2)/3$.} 
  If $\delta(H) \geq 2d(v) - n + 1$, then
  Lemma~\ref{lem:basicdelta} yields
  \[ \te(H) \leq (d(v) - \delta(H)) + \frac{d(v)(d(v) -
      \delta(H)-1)}{2} = n-1 + \frac{d(v)(n - d(v) - 4)}{2} < \frac{(n-1)^2}{9}, \]
  where the last inequality follows from the assumption that $d(v) \geq (2n+1)/3$.

  Hence we may assume that $\delta(H) = 2d(v)-n$. We consider two subcases:
  either $d(v) = (2n+1)/3$ or $d(v) = (2n+2)/3$.

  \caze{3a}{$d(v) = (2n+1)/3$.} Here $\delta(H) = 2d(v) - n$ gives
  $\delta(H) = (n+2)/3 = (d(v)+1)/2$. Hence, Lemma~\ref{lem:plushalf_plusone}
  yields
  \[ \te(H) \leq \frac{d(v)^2}{4} - \frac{d(v)}{2} + \frac{1}{4} = \frac{(n-1)^2}{9}, \]
  and so $\te(H) \leq k_3(T(n,3)) - k_3(T(n-1,3))$.

  This yields $\ttG \leq k_3(T(n,3))$ for the case $d(v) = (2n+1)/3$.
  To obtain the strict inequality $\ttG < k_3(T(n,3))$, suppose to the
  contrary that $\ttG = k_3(T(n,3))$. Since
  $\te(H) \leq k_3(T(n,3)) - k_3(T(n-1,3))$, we must also have
  $\ttee(G-v) = k_3(T(n-1,3))$. By the induction hypothesis,
  $G-v \iso T(n-1,3)$. In particular, $G-v$ is $K_4$-free. Let $\cee'$
  be a smallest $K_3$ clique cover in $G-v$. Since $G-v$ is
  $K_4$-free, every clique in $\cee'$ is a triangle.

  Since $\delta(H) = \frac{d(v)+1}{2} > \frac{d(v)}{2} = \sizeof{V(H)}/2$, every edge in $H$ is contained in some
  triangle of $H$. Since $G-v$ is $K_4$-free, every triangle of $H$ is contained
  in $\cee'$. Let $\cee$ be the collection of cliques obtained by replacing
  every triangle $T$ of $H$ contained in $\cee'$ with the clique $T \cup \{v\}$.
  Now $\cee$ covers all triangles in $G$, and
  \[ \sizeof{\cee} = \sizeof{\cee'} \leq k_3(T(n-1,3)) < k_3(T(n,3)). \]
  This contradicts the hypothesis that $\ttG = k_3(T(n,3))$. We conclude
  that $\ttG < k_3(T(n,3))$ when $d(v) = (2n+1)/3$.

  \caze{3b}{$d(v) = (2n+2)/3$.} Here $\delta(H) = 2d(v) - n$ gives $\delta(H) = \frac{n+4}{3} = \frac{d(v)}{2} + 1$.
  In this case, $n \equiv 2 \pmod{3}$. Since we have assumed also that $n > 3$, we have $n \geq 5$, so
  Lemma~\ref{lem:plushalf_plusone} yields
  \[ 
    \te(H) \leq \frac{d(v)^2}{4} - d(v) + 2 = \frac{n^2}{9} - \frac{4n}{9} + \frac{13}{9} \leq \frac{n^2 - n - 2}{9},
  \]
  where the inequality is strict for $n > 5$.  Thus,
  $\te(H) \leq k_3(T(n,3)) - k_3(T(n-1,3))$, and this inequality is
  strict for $n > 5$. On the other hand, when $n=5$, we have
  $\delta(G) = d(v) = 4$, so that $G$ is a complete graph, which
  forces $\ttG = 1 < k_3(T(5;3))$. Thus, when $d(v) = (2n+2)/3$, we
  have $\ttG < k_3(T(n,3))$.
	
  Thus, in all cases, $\ttG \leq k_3(T(n,3))$, and equality holds only
  in Case~1 when $G \iso T(n,3)$.
\end{proof}

Considering the way that Lemma~\ref{lem:basicdelta} is used in
  the proof of Theorem~\ref{thm:upper_bound}, one might hope to prove
  an analogous ``minimum-degree version'' of
  Theorem~\ref{thm:upper_bound} and then use it to prove the $t=4$
  case of Conjecture~\ref{conjecture}. Unfortunately, one runs
  into difficulties with this approach very quickly. The main
  difficulty is that the proof of Lemma~\ref{lem:basicdelta} relies
  very strongly on there only being two types of edges that must be
  covered: edges within a single $A_i$ and edges with one endpoint in
  $A_i$ and the other in $A_j$. When $t=3$, on the other hand, there
  are (at least) three possible types of $K_3$: those contained within
  a single $A_i$, those with two endpoints in one $A_i$ and the other
  in $A_j$, and those with endpoints in three different sets
  $A_i, A_j, A_k$. This makes it considerably more difficult to find a
  way to cover all copies of $K_3$ and efficiently count the number of
  cliques used in the process. As $t$ grows larger, even more
  configurations are possible, making this approach more difficult
  than anticipated.

The results stated in Theorem~\ref{thm:upper_bound} have a very close connection with their counterparts established for hypergraphs~\cite{Bollobas1977, Lehel1982}. In the following we discuss the similarity and the differences between our results and those known in the hypergraph literature. A $t$-uniform hypergraph $H = (V(H),E(H))$ consists of a vertex set $V(H)$ and a hyperedge set $E(H)$, where each hyperedge is a set of some $t$ vertices. A $2$-uniform hypergraph is
simply a graph. For $p \geq t \geq 2$, let $K_p^{(t)}$ denote the 
hyperclique on $p$ vertices, i.e., a set of $p$ vertices of
a hypergraph where every subset of $t$ vertices forms a hyperedge.
Note that the usual clique $K_t$ is the same as $K_t^{(2)}$. 
Let $h_t(n,p)$ denote the maximum number of \emph{hyperedges} that a $K_p^{(t)}$-free $t$-uniform
hypergraph on $n$ vertices can have. 
Let $k_t(n,p)$ denote the maximum number of $K_t$ in a $K_p$-free graph. 
It was proved by Moon and Moser~\cite{MoonMoser1962}, and by 
Sauer~\cite{Sauer1968} 
that $k_t(n,p)$ is precisely the number of $K_t$ in the \turan~graph $T(n,p-1)$.
In other words, $k_t(n,p) = k_t(T(n,p-1))$. 
The following result was conjectured by
\bollobas~\cite{Bollobas1977} and proved by Lehel~\cite{Lehel1982}.

\begin{theorem}[Lehel~\cite{Lehel1982}]
\label{thm:Lehel}
The edges of every $t$-uniform hypergraph of order $n$ can be covered by at most $h_t(n,p)$ edges and copies of $K_p^{(t)}$.
\end{theorem} 

When $t = 2$ and $p = 3$, we have $h_2(n,3) = k_2(n,3) = k_2(T(n,2)) = \left\lfloor \frac{n^2}{4}\right\rfloor$ (by Mantel's theorem and also by \turan~\cite{Turan1941}), and hence Theorem~\ref{thm:Lehel} reduces to the classical result on edge clique cover by \erdos, Goodman, and \posa~\cite{ErdosGoodmanPosa1966} mentioned in the introduction, which states that the edges of every graph on $n$ vertices can be covered by using at most $\left\lfloor \frac{n^2}{4}\right\rfloor$ edges and triangles.  

When $t = 3$, Theorem~\ref{thm:Lehel} states that 
one can use at most $h_3(n,4)$ hyperedges and $K_4^{(3)}$'s to cover all hyperedges in a $3$-uniform hypergraph. However, Theorem~\ref{thm:upper_bound} does not follow from this statement. Indeed, our theorem establishes that one can use at most $k_3(n,4)$ cliques to cover all triangles in any graph on $n$ vertices. We emphasize that $h_3(n,4)$ is strictly larger than $k_3(n,4)$. Recall that $k_3(n,4)$, which is the number of triangles in the \turan~graph 
$T(n,3)$, can be computed explicitly as $\lfloor n/3 \rfloor \lfloor (n+1)/3 \rfloor \lfloor (n+2)/3 \rfloor \approx n^3/27$.  
By contrast, the determination of $h_3(n,4)$, even asymptotically, has remained open since the original work of \turan~\cite{Turan1941}. Moreover, \turan~established that (see also~\cite{ChungLu1999})
\begin{equation}
\label{eq:t34}
h_3(n,4) \geq 
\begin{cases}
m^2(5m-3)/2, &\text{ if } n = 3m,
\vspace{0.1in} \\
m(5m^2+2m-1)/2, &\text{ if } n = 3m+1,
\vspace{0.1in} \\
m(m+1)(5m_2)/2, &\text{ if } n = 3m + 2. 
\end{cases}
\end{equation}
The hypergraph that gives rise to this lower bound is a modification of the usual \turan~graph $T(n, 3)$ to the hypergraph setting, which can be constructed as follows.
The vertex set is partitioned into three (almost) equal sets $V_1$, $V_2$, and $V_3$,
where $|V_1| = \lfloor n/3 \rfloor$, $|V_2| = \lfloor (n+1)/3 \rfloor$, and $|V_3| = \lfloor (n+2)/3 \rfloor$. The hyperedge set consists of the $3$-tuples $e = \{u,v,w\}$, where either $u \in V_1, v \in V_2, w \in V_3$, or $u, v \in V_i$ and $w \in V_{(i+1) \pmod 3}$. Notice that the first type of hyperedges corresponds to the triangles
of the \turan~graph $T(n,3)$, while the second type of hyperedges corresponds to new triangles that are unique to the hypergraph context. From \eqref{eq:t34}, $h_3(n,4)$ is in order of
$\frac{5}{2}\frac{n^3}{27}$, which is strictly larger than $k_3(n,4) \approx \frac{n^3}{27}$. In fact, \turan~conjectured that \eqref{eq:t34} is actually an equality, which was known to be true for all $n \leq 13$~\cite{Spencer1990}.  

The key point that leads to the difference between $h_3(n,4)$ and
$k_3(n,4)$ is that while a $K_4$-free graph (removing all edges that
do not belong to any triangles) can be considered as a $K_4^{(3)}$-free
$3$-uniform hypergraph, the converse is \emph{not} true. For example,
a hypergraph with the vertex set $\{u,v,w,x\}$ and the hyperedge set
$\{\{u,x,w\}, \{v,x,w\}$, $\{u,v,w\}\}$ is a $K_4^{(3)}$-free
hypergraph, but it corresponds exactly to a $K_4$ as a
graph. Therefore, counting edges in a $K_4^{(3)}$-free $3$-uniform
hypergraph is not the same as counting triangles in a $K_4$-free
graph. In fact, the maximum number of hyperedges in such
$K_4^{(3)}$-free hypergraphs is larger than the maximum number of
triangles in $K_4$-free graphs. That is the reason Theorem~\ref{thm:Lehel}
produces a strictly worse upper bound than the tight upper bound we
obtain in Theorem~\ref{thm:upper_bound}.  This is also evident from
the fact that while according to Theorem~\ref{thm:Lehel}, the
hyperedges of every $3$-uniform hypergraph can be covered by using at
most $h_3(n,4)$ $K_4^{(3)}$'s and hyperedges,
Remark~\ref{rm:clique_size} states that we cannot cover all triangles
in $K_n$ $(n \geq 18)$ by $k_3(n,4) = k_3(T(n,3))$ $K_4$'s and
triangles.

\begin{remark}
\label{rm:clique_size}
It is impossible to cover all triangles in $K_n$ $(n \geq 18)$ by only
$k_3(n,4)$ triangles and copies of $K_4$.  Indeed, since
$k_3(n,4) \approx \frac{n^3}{27}$ and each $K_4$ contains four
triangles, $k_3(n,4)$ triangles and copies of $K_4$ can cover at most
$\approx \frac{4}{27}n^3 < \frac{n^3}{6} \approx \binom{n}{3}$
triangles, which is the number of triangles in $K_n$, for $n$
sufficiently large. It can be easily verified that this conclusion
holds for $n \geq 18$. More generally, as $k_t(n,t+1)\approx \frac{n^t}{t^t}$ and $(t+1)\frac{n^t}{t^t} < \frac{n^t}{t!} \approx \binom{n}{t}$ for any fixed $t \geq 3$ and sufficiently large $n$, we cannot cover all $K_t$'s in $K_n$ by only $k_t(n,t+1)$ $K_t$'s and copies of $K_{t+1}$.  
\end{remark}

Note also that the key lemma (Lemma 4.3) in the proof of
Lehel~\cite{Lehel1982} fails if we try to adapt it to the setting of
graphs and triangles. The lemma states that for any $2 \leq t < p$,
every $t$-uniform hypergraph containing $m$ edges has a $K_p$-free
edge subset of cardinality at least $m/2$. However, its
graph-and-triangle version, which would state that every graph
containing $m$ triangles has a $K_4$-free subset of triangles of
cardinality at least $m/2$, is no longer correct. A counterexample is
the graph $K_5$, which contains exactly ten triangles, where we cannot
find any subset of five triangles that does not include a
$K_4$. Indeed, as established by Moon and Moser~\cite{MoonMoser1962}
and Sauer~\cite{Sauer1968}, the \turan~graph $T(5,3)$ is the
$K_4$-free graph that contains the largest number of triangles, which
is only four. Thus, a direct adaptation of Lehel's arguments to our
setting does not imply our result.

\section{Algorithmic Considerations}
Scheinerman and Trenk~\cite{scheinerman-trenk} developed an algorithm
which computes the edge clique cover number of a chordal graph $G$.
Our primary goal in this section is to generalize their algorithm to
the context of $K_t$ clique covers; however, we will also generalize
the algorithm in two other respects. These generalizations may be of
interest even for the original edge clique cover problem.

Our first generalization is to consider a weighted version of the edge
clique cover problem in which each $t$-clique $k$ has an integer
weight $w(k)$ specifying the number of times the clique needs to be
covered. While we are primarily concerned with the unweighted version
of the problem (equivalently, the case where all $t$-cliques have
weight $1$), the most natural recursive formulation of even the
unweighted version of the algorithm involves passing to subproblems in
which some $t$-cliques no longer need to be covered, which is
equivalent to giving those cliques weight $0$. (In fact, this
distinction already appears in the original formulation of Scheinerman
and Trenk~\cite{scheinerman-trenk}, in which edges are labeled as ``covered''
as the algorithm proceeds.) Since weighted subproblems
arise naturally even when solving the unweighted problem, we formulate
the algorithm for the weighted problem from the outset.

Our second generalization is to observe that the Scheinerman--Trenk
algorithm works on a slightly more general class of graphs than the
chordal graphs, which we dub \emph{semichordal graphs}, and so in the
interest of generality we state the $K_t$ clique cover version of the
algorithm in terms of semichordal graphs rather than chordal graphs.

The remainder of this section consists of three subsections.
In the first subsection, we define the class of semichordal graphs
and discuss some of their properties. In the second subsection, we give
our generalization of the Scheinerman--Trenk algorithm and prove its
correctness. In the third subsection, we prove that the $K_t$ clique
cover problem on general graphs is NP-hard, justifying the development
of algorithms on specialized graph classes.

\subsection{A Superclass of Chordal Graphs}\label{sec:superclass}
A graph is said to be \emph{chordal} if it has no induced cycle of
length greater than $3$. A well-known characterization of chordal
graphs, due to Dirac~\cite{dirac}, is that they are the graphs which
admit a \emph{simplicial elimination ordering}: an ordering $v_1,
\ldots, v_n$ of the vertices of $G$ such that $N(v_i) \cap \{v_{i+1},
\ldots, v_n\}$ is a clique for each $i$.

The notion of a perfect elimination order admits a natural generalization,
as follows.
\begin{definition}[Aboulker--Charbit--Trotignon--Vu\v skovi\'c \cite{ACTV}]
  Let $\fee$ be a set of graphs. An \emph{$\fee$-elimination ordering}
  of a graph $G$ is an ordering $v_1, \ldots, v_n$ of the vertices of
  $G$ such that for each $i$, the induced subgraph $G[N(v_i) \cap
  \{v_{i+1}, \ldots, v_n\}]$ has no induced subgraph isomorphic to a
  graph in $\fee$.
\end{definition}
Thus, Dirac's result states that the chordal graphs are precisely the
graphs that admit a $\{\overline{K_2}\}$-elimination ordering. In Section~\ref{sec:weightkt},
we give an algorithm for computing weighted $K_t$ clique covers on a superclass of the
chordal graphs, defined as follows.
\begin{definition}
  A graph $G$ is \emph{semichordal} if it admits a
  $\{P_3\}$-elimination ordering, where $P_3$ is the path on three
  vertices.
\end{definition}
Equivalently, a graph is semichordal if it admits a vertex ordering
$v_1, \ldots, v_n$ such that for each $i$, the subgraph induced by
$N(v_i) \cap \{v_{i+1}, \ldots, v_n\}$ is a disjoint union of complete
graphs.  Since $\overline{K_2}$ is an induced subgraph of $P_3$,
Dirac's characterization immediately implies that every chordal graph
is semichordal. On the other hand, any cycle $C_n$ for $n > 3$ is a
semichordal graph that is not chordal.

As semichordal graphs are defined in terms of the existence of a
certain elimination ordering, it would be desirable to have a
characterization of these graphs in terms of their forbidden induced
subgraphs, analogous to the definition of chordal
graphs. Unfortunately, we are aware of no such characterization. The
following sufficient (but not necessary) condition for a graph to be
semichordal was discovered by Aboulker, Charbit, Trotignon, and Vu\v
skovi\'c~\cite{ACTV}.
\begin{definition}[\cite{ACTV}]
  A graph $G$ is a \emph{wheel} if there is a vertex $v$ of degree
  at least $3$ such that $G-v$ is isomorphic to a cycle. The vertex
  $v$ is the \emph{center} of the wheel and the subgraph $G-v$ is the
  \emph{rim} of the wheel. A wheel is a \emph{$3$-wheel} if there are
  three consecutive vertices $x,y,z$ on the rim such that the center
  is adjacent to $x$, $y$, and $z$.
\end{definition}
\begin{theorem}[\cite{ACTV}]
  If $G$ has no induced subgraph isomorphic to a $3$-wheel, then $G$
  is semichordal.
\end{theorem}
In fact, \cite{ACTV} proves that if $G$ has no induced subgraph
isomorphic to a $3$-wheel, then $G$ satisfies a stronger property
guaranteeing that a $\{P_3\}$-elimination ordering can be easily
found. We refer the reader to \cite{ACTV} for more details.
\subsection{Weighted Edges}\label{sec:weightkt}
In this section, we consider a weighted variant of the $K_t$ clique
cover problem.  Given a graph $G$, we assume that each $t$-clique $k
\subset V(G)$ is assigned a weight $w(k)$ representing the number of
times that $S$ must be covered. Our goal is to find a multiset $\cee$
of cliques in $G$ such that each $t$-clique $k$ is covered at least
$w(k)$ times. We formalize these notions as follows.
\begin{definition}
  Given a graph $G$ and an integer $t \geq 0$, let $\kay(G)$ be the
  family of all cliques in $G$, and let $\kay_t(G)$ be the family of
  all $t$-cliques in $G$.  Let $w : \kay_t \to \ints_{\geq 0}$ be a
  weight function on the $t$-cliques of $G$. A \emph{$(w,K_t)$-cover}
  of $G$ is a function $f : \kay(G) \to \nonneg$ such that
  $\sum_{k \subset K}f(K) \geq w(k)$ for all $k \in \kay_t(G)$, where the sum
  ranges over all $K \in \kay(G)$ with $k \subset K$.  When $f$ is a
  $(w,K_t)$-cover, we write $\cost(f)$ for the sum $\sum_{K \in
    \kay(G)}f(K)$.  The \emph{$(w,K_t)$-cover number} of $G$, written
  $\iw(G)$, is the minimum value of $\cost(f)$ over all
  $(w,K_t)$-covers of $G$.
\end{definition}
Observe that when $w(S) = 1$ for all $S \in \kay_t$, the
$(w,K_t)$-cover number of $G$ is just the $K_t$ clique cover number of
$G$.  We also define a corresponding dual problem.
\begin{definition}
 The \emph{$(w,K_t)$-clique packing number}
of $G$, written $\pw(G)$, is the optimum value of the following integer program:
\begin{align*}
  \text{maximize}\sum_{k \in \kay_t(G)}w(k)y(k),& \text{ subject to} \\
  \sum_{k \subset K}y(k) &\leq 1, \text{ for all $K \in \kay(G)$}, \\
  y(k) &\geq 0, \text{for all $k \in \kay_t(G)$} \\
  y(k) &\in \ints.
\end{align*}
A feasible solution to this integer program is called a
\emph{$(w,K_t)$-packing}. When $y$ is a $(w,K_t)$-packing write
$\valu(y)$ for $\sum_{k \in \kay_t(G)}w(k)y(k)$. (In some circumstances
it may be ambiguous which weight function is used to calculate $\valu(y)$, in
which case we write $\valu_w(y)$ to specify the weight function being used.)
\end{definition}

Let $\iw^*(G)$ and $\pw^*(G)$ denote the fractional relaxations
of $\iw(G)$ and $\pw(G)$, respectively. Standard LP duality gives
\[ \pw(G) \leq \pw^*(G) = \iw^*(G) \leq \iw(G). \]
We wish to show that when $G$ is semichordal, equality holds throughout.
  \begin{algorithm}
    \caption{Recursive algorithm $\optpair$ to produce a pair $(f,y)$,
      where $f$ is an optimal $(w,K_t)$-cover and $y$ is an
      optimal $(w,K_t)$-packing.\label{alg:optpair}}
    \begin{algorithmic}
      \IF{$G$ has no edges}
      \STATE{Return $(e, e)$, where $e$ is the
        empty function. }
      \ELSE
      \STATE{Let $v_1, \ldots, v_n$ be a
        $\{P_3\}$-elimination ordering of $G$.}
      \STATE{Let
        $Q_1, \ldots, Q_h$ be the components of $G[N(v_1)]$ of size at least $t-1$.}
      \COMMENT{Each $Q_i$ is a clique.}
      \STATE{Let $G' = G-v_1$.}
      \FORALL{$i \in \{1, \ldots h\}$}
      \STATE{Let $Q^*_i = \{v_1\} \cup Q_i$.}
      \STATE{Pick $Z_i \in \kay_{t-1}(Q_i)$ to maximize $w(\{v_1\} \cup Z_i)$ and let $t_i = w(\{v_1\} \cup Z_i)$.}
      \ENDFOR
      \STATE{Let $\lay = \kay_t(Q_1) \cup \cdots \cup \kay_t(Q_h)$.}
      \STATE{Let $w'(k) = w(k)$ for $k \in \kay_t(G') - \lay$ and let $w'(k) = \max\{0, w(k)-t_i\}$ for $k \in \kay_t(Q_i)$.}
      \STATE{Let $(f', y') = \optpair(G', w')$.}
      \STATE{Let $f(K) = f'(K)$ for $K \in \kay(G')$.}
      \STATE{Let $y(k) = y'(k)$ for $k \in \kay_t(G')$.}
      \COMMENT{$f$ and $y$ are only partially defined so far}
      \FORALL{$i \in \{1, \ldots, h\}$}
      \STATE{Let $f(Q^*_i) = t_i$.}
      \STATE{Let $y(\{v_1\} \cup Z) = 0$ for all $Z \in \kay_{t-1}(Q_i) - Z_i$.}
      \IF{$t_i=0$ or $y'(k) > 0$ for some $k \in \kay_t(Q_i)$}
      \STATE{Let $y(\{v_1\} \cup Z_i) = 0$.}
      \ELSE
      \STATE{Let $y(\{v_1\} \cup Z_i) = 1$.}
      \ENDIF
      \ENDFOR
      \STATE{Let $f(Q) = 0$ for all cliques $Q$ on which $f$ is not yet defined.}      
      \STATE{Return $(f,y)$.}
      \ENDIF
    \end{algorithmic}
  \end{algorithm}
\begin{theorem}
  If $G$ is semichordal, then for all $t \geq 1$ and all $w : \kay_t(G) \to \nonneg$,
\[ \iw(G) = \pw(G). \]
\end{theorem}
\begin{proof}
  We adapt the argument of Scheinerman and Trenk~\cite{scheinerman-trenk}.
  We claim that Algorithm~\ref{alg:optpair} produces a $(w,K_t)$-cover $f$
  and a $(w,K_t)$-packing $y$ such that $\cost(f) = \valu(y)$,
  and such that $y(k) > 0$ only if $w(k) > 0$.

  Our proof proceeds by induction on $\sizeof{E(G)}$. When
  $\sizeof{E(G)} = 0$ it is clear that the pair of empty functions
  $(e,e)$ returned by Algorithm~\ref{alg:optpair} has the desired
  properties.

  Now suppose that $\sizeof{E(G)} > 0$, let $v_1, \ldots, v_n$ be the
  $\{P_3\}$-elimination ordering used in Algorithm~\ref{alg:optpair},
  and let $(f', y') = \optpair(G', w')$. By the induction hypothesis,
  $f'$ and $y'$ are feasible for their respective integer programs,
  and $\cost(f') = \valu(y')$.

  First we argue that $f$ is feasible. First observe that
  $\kay_t(G) = (\kay_t(G') - \lay) \cup (\kay_t(Q_1^*) \cup \cdots
  \cup \kay_t(Q_h^*))$, and that $f(K) = f'(K)$ for all
  $k \in \kay(G')$.

  For all $k \in \kay_t(G') - \lay$, we have $w'(k) = w(k)$, and the feasibility of $f'$
  implies that
  \[ \sum_{\substack{k \subset K\\ K \in \kay(G)}}f(K) \geq \sum_{\substack{k \subset K \\ K \in \kay(G')}}f'(K) \geq w'(k) =
    w(k) \] for all $k \in \kay_t(G')- \lay$.

  On the other hand, for $k \in \kay_t(Q_i)$, we have
  $w(k) \leq w'(k) + t_i$.  Since $f'$ is feasible and
  $f(Q^*_i) = t_i$, for these $t$-cliques we have
  \[ \sum_{\substack{k \subset K\\ K \in \kay(G)}}f(K) = f(Q^*_i) + \sum_{\substack{k \subset K\\ K \in \kay(G')}}f'(K) \geq t + w'(k) \geq w(k). \]
  
  The remaining $t$-cliques to consider are those in $\kay_t(Q^*_i) \setminus \kay_t(Q_i)$
  for some $i$, that is, the $t$-cliques containing $v_1$. Any such $t$-clique $k$ is of
  the form $k = \{v_1\} \cup Z$ for some $Z \in \kay_{t-1}(Q_i)$, and is contained in the
  clique $Q^*_i$. Since $f(Q^*_i) = t_i = \max_{Z}(w(\{v_1\} \cup Z))$, where the maximum
  is taken over all $Z \in \kay_{t-1}(Q_i)$, we see that for such $k$,
  \[ \sum_{\substack{k \subset K\\ K \in \kay(G)}} f(K) \geq f(Q^*_i) \geq w(k), \]  
  and so $f$ is feasible.

  Now we argue that $y$ is feasible. Consider any clique $K \in \kay(G)$. If $v_1 \notin K$, then $K \in \kay(G')$ and so $y(k) = y'(k)$
  for all $t$-cliques $k \subset K$, so by the induction hypothesis, we have
  \[ \sum_{k \subset K} y(k) = \sum_{k \subset K} y'(k) \leq 1. \]
  Thus, we may assume that $v_1 \in K$. This implies that $K \subset Q^*_i$
  for some $i$. Observe that
  \begin{align*}
    \sum_{k \subset K} y(k) &\leq \sum_{Z \in \kay_{t-1}(Q_i)}y(\{v_1\} \cup Z) + \sum_{k  \in \kay_t(Q_i)}y(k) \\
    &= y(\{v\} \cup Z_i) + \sum_{k  \in \kay_t(Q_i)}y'(k),
  \end{align*}
  where $\sum_{k \in \kay_t(Q_i)}y'(k) \leq 1$ by the feasibility of $y'$.
  Thus, if $y(\{v_1\} \cup Z_i) = 0$, then the constraint for $K$ is satisfied.
  The only way the algorithm allows $y(\{v_1\} \cup Z_i) > 0$ is when
  $y'(k) = 0$ for all $k \in \kay_t(Q_i)$, in which case the constraint
  is again satisfied.

  Next we argue that $y(k) > 0$ only if $w(k) > 0$. Consider any
  $k \in \kay_t(G)$ with $y(k) > 0$. If $v_1 \notin k$, then
  $k \in \kay_t(G')$, so the induction hypothesis implies that
  $w'(k) > 0$. Since $w'(k) \leq w(k)$, this implies that $w(k) > 0$
  as well. On the other hand, if $v_1 \in k$, then $y(k) > 0$ is only
  possible if $k = \{v_1\} \cup Z_i$ for some $i$ with $t_i >
  0$. Since $t_i = w(\{v_1\} \cup Z_i) = w(k)$, we again see that
  $y(k) > 0$ implies $w(k) > 0$.
  
  Finally we argue that $\cost(f) = \valu(y)$. Let $R$ be the set
  the of indices $i$ such that $y(\{v_1\} \cup Z_i) = 1$.
  Observe that
  \[
    \cost(f) = \cost(f') + \sum_{i \in R}f(Q_i^*) = \cost(f') + \sum_{i=1}^ht_i.
  \]
  By the induction hypothesis, $\cost(f') = \valu_{w'}(y')$. We wish
  to determine $\valu_{w}(y')$. Observe that
  \[ \valu_{w}(y') - \valu_{w'}(y') = \sum_{y'(k) = 1}[ w(k) - w'(k) ], \]
  and by the induction hypothesis, $y'(k) = 1$ implies that $w'(k) > 0$,
  so that $w'(k) = w(k) - t_i$. Hence,
  \[ \valu_{w}(y') - \valu_{w'}(y') = \sum_{y'(k) = 1}t_i = \sum_{i \notin R} t_i, \]  
  where the last equality holds because $i \notin R$ implies that either
  $t_i = 0$ or that $y'(k) > 0$ for some $k \in \kay_t(Q_i)$, in which case
  feasibility of $y'$ implies that there is exactly one $k$ for which this is true.
  The remaining cliques to count in $\valu(y)$ are the cliques $\{v_1\} \cup Z_i$
  where $i \in R$. Thus,
  \[ \valu(y) = \valu_{w}(y') + \sum_{i \in R} t_i = \valu_{w'}(y') + \sum_{i=1}^h t_i
    = \cost(f') + \sum_{i=1}^h t_i = \cost(f), \]
  as desired.
\end{proof}
\subsection{NP-hardness of $K_t$ clique cover problem on general graphs}
\label{sec:complexity}
In this section, we will prove that for any fixed $t \geq 1$, the $K_t$
clique cover problem is NP-hard. This justifies developing algorithms
to solve this problem on specialized graph classes such as the semichordal
graphs, since (unless $P=NP$) there can be no polynomial-time algorithm
to solve the problem on general graphs.

Kou, Stockmeyer, and Wong~\cite{KouStockmeyerWong78} proved that the
edge clique cover problem, i.e., the $K_2$ clique cover problem, is
NP-hard, via a reduction from the $K_1$ clique cover problem, which is
equivalent to vertex coloring in the complementary graph and thus
NP-hard. We generalize their approach, reducing the $K_{t-1}$ clique
cover problem to the $K_t$ clique cover problem for each $t \geq 2$,
which implies that each of these problems is NP-hard. We
formalize that $K_t$ clique cover problem as a decision problem
as follows:
\begin{center}
  \begin{tabular}{l}
    \hline
    $\KCC(t)$\\
    \quad \textbf{Input}: a graph $G$, and a number $k$;\\
    \quad \textbf{Output}: YES if $\tKtG \leq k$ and NO otherwise;\\
    \hline
  \end{tabular}  
\end{center}
\begin{proposition}
The decision problem $\KCC(t)$ is NP-complete for any constant $t \geq 1$.
\end{proposition}
\begin{proof}
  We adapt the proof of Kou, Stockmeyer, and Wong.  It is obvious that
  $\KCC(t)$ is in NP.  We prove the NP-completeness of this problem by
  induction on $t$.  It is known that $\KCC(1)$ is
  NP-complete~\cite{Karp1972}.  Now suppose that $t \geq 2$ and that
  $\KCC(t-1)$ is NP-complete. We aim to show that $\KCC(t)$ is also
  NP-complete.

Let $G$ be an arbitrary graph of order $n$ and $k \geq 1$, and (mirroring the notation of Kou--Stockmeyer--Wong for the case $t=2$)
let $e = \sizeof{\kay_t(G)}$. Let $G'$ be the graph obtained from $G$ by introducing
\begin{itemize}
\item $s = 1 + e$ new vertices $\{u_1,\ldots, u_s\}$, and
\item $sn$ new edges that connect the new vertices to all existing vertices of $G$.
\end{itemize}
Since there are at most ${n \choose t}$ possible $t$-cliques in
$\kay_t(G)$, it is clear that this construction can be carried out in
polynomial time.  (Since $t$ is fixed in the decision problem
$\KCC(t)$, it does not matter that the degree of the polynomial
depends on $t$.)  Let $k' = sk + e$. We demonstrate that
$\theta_{K_{t-1}}(G) \leq k$ if and only if $\tKt(G') \leq k'$.

We first claim that if $\theta_{K_{t-1}}(G) \leq k$, then $\tKt(G') \leq k'$.
Suppose that $\A$ is a $K_{t-1}$ clique cover in $G$, with $\sizeof{\A} \leq k$.
For each $i \in \{1, \ldots, s\}$, let $\B_i = \{u_i \cup S \st S \in \A\}$,
and let $\B = \B_1 \cup \cdots \cup \B_s$. Now $\B$ covers every $t$-clique in $G'$
except perhaps for some $t$-cliques totally contained in $G$. Adding each such clique
to $\B$ separately yields a $K_t$ clique cover in $G'$ having at most $sk + e$
cliques, so that $\tKt(G') \leq sk+e = k'$.

It remains to show that if $\tKt(G') \leq k'$, then
$\theta_{K_{t-1}}(G) \leq k$. Suppose that $\D$ is a $K_t$ clique
cover in $G'$ with $\sizeof{\D} \leq k'$. For each
$i \in \{1, \ldots, s\}$, let $\D_i$ be the subset of cliques in $\D$
that contain the vertex $u_i$. Observe that for $i \neq j$, the vertices $u_i$ and $u_j$ are not adjacent,
so that $\D_i$ and $\D_j$ are disjoint. Hence,
\[
\sum_{i = 1}^s |\D_i| \leq |\D| \leq k'.
\]
Therefore, if $i_{\min}$ is an index such that $|\D_{i_{\min}}| = \min_{1 \leq i \leq s} |\D_i|$, then 
\[
|\D_{i_{\min}}| \leq \left\lfloor \dfrac{\sum_{i = 1}^s |\D_i|}{s} \right\rfloor \leq \left\lfloor \dfrac{k'}{s}
\right\rfloor
=\left\lfloor \dfrac{sk + e}{s} \right\rfloor = k,
\]
where the last equality holds because $s = 1 + e$. 
Then, by removing $u_{i_{\min}}$ from all cliques in $\D_{i_{\min}}$, we obtain a $K_{t-1}$ clique
cover of $G$ of size at most $k$. The proof follows. 
\end{proof}
We note briefly that Kou, Stockmeyer, and Wong~\cite{KouStockmeyerWong78} actually
proved a stronger property than NP-completeness of the edge clique cover problem:
they used a result of Garey and Johnson~\cite{GareyJohnsonNearOptimal}
on the inapproximability of the $K_1$ clique cover problem to prove
that if $P \neq NP$, then there is no $c$-approximation algorithm for the
edge clique cover problem for any $c < 2$. We believe that this inapproximability
proof extends to the $K_t$ clique cover problem by making the same modifications we made
to the NP-hardness proof, but in the interest of simplicity, we have chosen only
to present the NP-hardness version of the proof.

\section*{Acknowledgment}
This work was supported by the NSF grants 1527636 and 1527636. We thank the anonymous
referees for their careful reading and their helpful comments which improved the
presentation of the paper.

\bibliographystyle{plain}
\bibliography{TriangleCliqueCover,bibchordal}

\end{document}